\documentclass[a4paper]{amsart}
\usepackage[margin=3.5cm]{geometry}
\usepackage[utf8]{inputenc}
\usepackage{amsmath,amsfonts,amssymb,amsthm, multirow}
\usepackage{comment, stackrel}
\usepackage{hyperref}
\usepackage{makecell}
\usepackage{hhline}
\usepackage{tikz}
\usetikzlibrary{calc}

\theoremstyle{plain}
\newtheorem{theorem}{Theorem}[section]

\newtheorem{corollary}[theorem]{Corollary}
\newtheorem{lemma}[theorem]{Lemma}

\renewcommand{\phi}{\varphi}

\theoremstyle{definition}
\newtheorem{definition}[theorem]{Definition}

\theoremstyle{remark}

\newcommand{\C}{\mathbb{C}}                    

\newcommand{\g}{\mathfrak{g}}
\newcommand{\CA}{\mathcal{A}}

\newcommand{\tr}{\operatorname{tr}}

\title{Hom--Lie Algebras and Explicit MSS Partition Bounds for $(6,3)$ Biangular Frames}

\author{Abraham Orinda}
\address{Department of Mathematics and Statistics, Bowling Green State University, Ohio}
\email{aorinda@bgsu.edu}

\author{Geoffrey Mboya}
\address{Department of Mathematics, Maseno University, Maseno, Kenya}
\email{gomboya@maseno.ac.ke}

\author{Arvince Ogendi}
\address{Department of Mathematics, Defence Forces Technical College, Nairobi, Kenya}
\email{aogendi@ndu.ac.ke}

\begin{document}
	\begin{abstract}
		We study the algebraic structure of $(6,3)$ biangular Parseval frames. The two-distance property yields adjacency matrices $A_1,A_2$ whose span forms a three-dimensional commutative algebra, and adjoining the commutator $[A_1,A_2]$ produces a three-dimensional Lie algebra $\g$. The Gram matrix $G = I + c_1 A_1 + c_2 A_2$ induces a derivation $\alpha(X) = [G,X]$ on $\g$, equipping $(\g, [\cdot,\cdot], \alpha)$ with a Hom--Lie algebra structure. We compute all structure constants explicitly in terms of the strongly regular graph parameters $(k_1,\lambda_1,\mu_1,k_2,\lambda_2,\mu_2)$ and the frame angles $(c_1,c_2)$, and use this framework to derive explicit bounds for the partial frame operators arising in Marcus--Spielman--Srivastava (MSS) partitions. These bounds depend directly on the structure constants and refine the universal MSS estimate in the regime of highly unbalanced partitions.
	\end{abstract}

	\maketitle
	
	\section{Introduction}
	\label{sec:introduction}
	
	The Marcus--Spielman--Srivastava (MSS) theorem~\cite{MSS} guarantees the existence of partitions of Parseval frames with controlled operator norms. For a frame $\{v_j\}_{j=1}^M \subset \C^m$ satisfying $\|v_j\|^2 \le \delta$ and any integer $r$, there exists a partition $\{S_1,\dots,S_r\}$ of $\{1,\dots,M\}$ such that
	\begin{equation}
		\Bigl\|\sum_{j\in S_k} v_jv_j^*\Bigr\| \le \Bigl(\sqrt{\tfrac1r}+\sqrt{\delta}\Bigr)^2\qquad (k=1,\dots,r). \label{eq:mss}
	\end{equation}
	While this bound is universal, frames with additional structure may admit sharper estimates. We consider the simplest nontrivial case: $M=6$, $m=3$, with the frame constrained to be {\em biangular}, meaning the absolute inner products between distinct vectors take only two values $c_1$ and $c_2$.
	
	Such frames carry a rich combinatorial structure: the pattern of which pair takes which angle defines a two-distance graph on six vertices, which for a Parseval frame must be strongly regular, with parameters that determine the frame geometry completely. Moreover, the commutator of the two adjacency matrices generates a three-dimensional Lie algebra, and the Gram matrix induces a natural derivation on it, yielding a Hom--Lie algebra structure.
	
	Our contribution is threefold: we construct a Hom--Lie algebra from the frame data, with the twist map $\alpha$ arising from the Gram matrix; we compute all structure constants explicitly in terms of the strongly regular graph parameters; and we use this algebraic framework to derive explicit MSS-type bounds that depend on the structure constants, comparing them with the universal MSS estimate.
	
	Section~\ref{sec:preliminaries} recalls background on biangular Parseval frames. Section~\ref{sec:graph} analyzes the graph structure and adjacency algebra. Section~\ref{sec:lie} constructs the Lie algebra and computes its structure constants. Section~\ref{sec:hom-lie} equips this Lie algebra with the Hom--Lie structure induced by the Gram matrix. Sections~\ref{sec:norm-bounds} and~\ref{sec:bounds} derive explicit bounds for MSS partitions using this algebraic framework and compare them numerically with the universal estimate. Section~\ref{sec:conclusion} concludes.
	
	\section{Preliminaries on Biangular Parseval Frames}
	\label{sec:preliminaries}
	
	\begin{definition}
		A finite set $\Phi=\{v_1,\dots,v_M\}\subset \C^m$ is an $\frac{M}{m}$-tight {\em Parseval frame} if
		\[
		\sum_{j=1}^M v_jv_j^* = \frac{M}{m}I_m.
		\]
	\end{definition}
	For a Parseval frame, the norms are constant: $\|v_j\|^2 = \frac{m}{M} =: \delta$.
	
	\begin{definition}
		A Parseval frame is {\em biangular} if there exist constants $c_1,c_2\ge 0$ such that $|\langle v_i,v_j\rangle| \in \{c_1,c_2\}$ for all $i\neq j$.
	\end{definition}
	
	We specialize to the case $M=6$, $m=3$, so $\delta = 1/2$. Let $G$ be the Gram matrix with entries $G_{ij} = \langle v_i,v_j\rangle$. For a Parseval frame,
	\begin{equation}
		G^2 = \frac{M}{m}G = 2G. \label{eq:gram-square}
	\end{equation}
	Thus $G$ has eigenvalues $2$ (algebraic multiplicity $m=3$) and $0$ (algebraic multiplicity $M-m=3$).
	
	\section{Graph Structure and Adjacency Algebra}
	\label{sec:graph}
	
	Define relations on $\{1,\dots,6\}$ by $i\sim_1 j \iff |\langle v_i,v_j\rangle| = c_1$ and $i\sim_2 j \iff |\langle v_i,v_j\rangle| = c_2$, and let $A_1,A_2$ be the corresponding adjacency matrices. Since every distinct pair is of one type or the other,
	\[
	A_1 + A_2 = J - I,
	\]
	where $J$ is the all-ones matrix. The Gram matrix can be written as
	\begin{equation}
		G = I + c_1 A_1 + c_2 A_2. \label{eq:gram-matrix}
	\end{equation}
	Consider the vector space $\CA = \operatorname{span}\{I, A_1, A_2\}$. Since $A_1$, $A_2$, and $J$ are linearly independent for a nontrivial two-distance graph, $\dim\CA = 3$.
	
	\subsection{Strongly regular graph parameters}
	\label{subsec:srg}
	
	For a $(6,3)$ biangular Parseval frame, the underlying graph must be strongly regular, with parameters:
	\begin{itemize}
		\item $k_1$: valency of $A_1$ (number of type-1 neighbors per vertex);
		\item $\lambda_1$: number of type-1 common neighbors for type-1 adjacent vertices;
		\item $\mu_1$: number of type-2 common neighbors for type-1 adjacent vertices;
		\item $k_2$: valency of $A_2$ (number of type-2 neighbors per vertex);
		\item $\lambda_2$: number of type-2 common neighbors for type-2 adjacent vertices;
		\item $\mu_2$: number of type-1 common neighbors for type-2 adjacent vertices.
	\end{itemize}
	These parameters satisfy the consistency relations
	\[
	k_1 + k_2 = 5, \qquad \lambda_1 + \mu_1 = k_1 - 1, \qquad \lambda_2 + \mu_2 = k_2 - 1.
	\]
	For a $(6,3)$ frame, the only possible nontrivial strongly regular graphs are the complement of a perfect matching and the complete bipartite graph $K_{3,3}$. We work with the generic case
	\begin{equation}
		k_1 = 2,\quad \lambda_1 = 0,\quad \mu_1 = 1,\quad k_2 = 3,\quad \lambda_2 = 1,\quad \mu_2 = 1. \label{eq:srg-params}
	\end{equation}
	
	\begin{figure}[ht]
		\centering
		\begin{tikzpicture}[
			vertex/.style={circle, draw=black!70, fill=black!10,
				inner sep=2pt, minimum size=5mm, font=\small},
			edge1/.style={draw=blue!80!black, line width=1.2pt},
			edge2/.style={draw=red!80!black, line width=1.2pt, dashed}
			]
			\foreach \i [count=\x] in {1,...,6} {
				\node[vertex] (v\i) at ({90-60*(\x-1)}:2cm) {$v_\i$};
			}
			\draw[edge1] (v1) -- (v3);
			\draw[edge1] (v1) -- (v4);
			\draw[edge1] (v2) -- (v5);
			\draw[edge1] (v2) -- (v6);
			\draw[edge1] (v3) -- (v5);
			\draw[edge1] (v4) -- (v6);
			\draw[edge2] (v1) -- (v2);
			\draw[edge2] (v1) -- (v5);
			\draw[edge2] (v1) -- (v6);
			\draw[edge2] (v2) -- (v3);
			\draw[edge2] (v2) -- (v4);
			\draw[edge2] (v3) -- (v4);
			\draw[edge2] (v3) -- (v6);
			\draw[edge2] (v4) -- (v5);
			\draw[edge2] (v5) -- (v6);
			\node[draw, rounded corners, fill=white, inner sep=3pt, anchor=north]
			at (0,-2.5) {
				\begin{tabular}{l}
					Type-1: $c_1 = 1/2$, $k_1=2$ (solid) \\
					Type-2: $c_2 = 1/4$, $k_2=3$ (dashed)
				\end{tabular}
			};
		\end{tikzpicture}
		\caption{Strongly regular graph underlying the $(6,3)$ biangular Parseval frame. Each vertex has two type-1 neighbors (solid) and three type-2 neighbors (dashed), giving $k_1=2$, $\lambda_1=0$, $\mu_1=1$, $k_2=3$, $\lambda_2=1$, $\mu_2=1$: the complement of a perfect matching on six vertices.}
		\label{fig:srg}
	\end{figure}
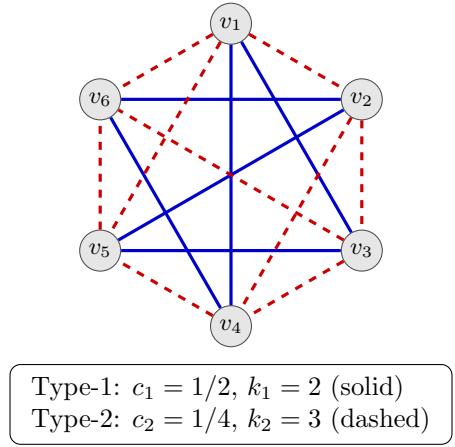
	
	\subsection{Algebraic relations}
	\label{subsec:algebraic}
	
	Using standard strongly regular graph identities,
	\begin{align}
		A_1^2 &= k_1 I + \lambda_1 A_1 + \mu_1 A_2, \label{eq:srg-id1}\\
		A_2^2 &= k_2 I + \lambda_2 A_2 + \mu_2 A_1. \label{eq:srg-id2}
	\end{align}
	
	\subsection{Determination of frame angles}
	\label{subsec:angles}
	
	The frame angles are not independent of the graph parameters. Substituting~\eqref{eq:srg-id1} and~\eqref{eq:srg-id2} into $G^2 = 2G$ from~\eqref{eq:gram-square}, and using~\eqref{eq:gram-matrix}, yields constraints on $c_1$ and $c_2$. For the parameters~\eqref{eq:srg-params} we obtain the unique positive solution
	\begin{equation}
		c_1 = \frac{1}{2},\quad c_2 = \frac{1}{4}. \label{eq:angles}
	\end{equation}
	Thus, for this family, the frame geometry is completely determined.
	
	\section{Lie Algebra Structure}
	\label{sec:lie}
	
	Define the vector space $\g = \operatorname{span}\{e_1, e_2, e_3\}$, where
	\[
	e_1 = A_1,\quad e_2 = A_2,\quad e_3 = [A_1, A_2] = A_1A_2 - A_2A_1,
	\]
	equipped with the commutator bracket $[X,Y] = XY - YX$.
	
	\begin{theorem} $(\g, [\cdot,\cdot])$ is a three-dimensional Lie algebra.
	\end{theorem}
	\begin{proof} By definition, $[e_1,e_2] = e_3 \in \g$. It remains to show that $[e_1,e_3]$ and $[e_2,e_3]$ lie in $\g$. Using~\eqref{eq:srg-id1} and~\eqref{eq:srg-id2} together with the explicit values~\eqref{eq:srg-params} and~\eqref{eq:angles}, a direct computation shows that both commutators are linear combinations of $e_1, e_2, e_3$. Hence $\g$ is closed under the bracket.
	\end{proof}
	
	\subsection{Structure constants}
	\label{subsec:structure}
	
	Write the brackets as
	\[
	[e_1,e_2] = e_3, \qquad [e_1,e_3] = a_1 e_1 + a_2 e_2 + a_3 e_3, \qquad [e_2,e_3] = b_1 e_1 + b_2 e_2 + b_3 e_3.
	\]
	
	\begin{theorem}[Explicit structure constants]\label{thm:structure}
		For the $(6,3)$ biangular Parseval frame with parameters~\eqref{eq:srg-params} and angles~\eqref{eq:angles}, the structure constants are
		\[
		a_1 = -\frac{16}{7}, \quad a_2 = 0, \quad a_3 = 1, \qquad
		b_1 = \frac{1}{2}, \quad b_2 = \frac{16}{7}, \quad b_3 = -\frac{1}{4}.
		\]
	\end{theorem}
	\begin{proof} Using~\eqref{eq:srg-id1} and~\eqref{eq:srg-id2} with the numerical values~\eqref{eq:srg-params} and~\eqref{eq:angles}, we compute
		\[
		[e_1,e_3] = [A_1, [A_1,A_2]] = A_1^2A_2 - 2A_1A_2A_1 + A_2A_1^2.
		\]
		Substituting $A_1^2 = 2I + A_2$ and $A_2^2 = 3I + A_2 + A_1$, and using $A_1A_2 + A_2A_1 = -\frac{1}{2}I - \frac{1}{2}A_1 - \frac{5}{2}A_2$, obtained from the Gram condition, we get after simplification
		\[
		[e_1,e_3] = -\frac{16}{7}A_1 + [A_1,A_2],
		\]
		so $a_1 = -16/7$, $a_2 = 0$, $a_3 = 1$. The computation for $[e_2,e_3]$ is analogous, yielding $b_1 = 1/2$, $b_2 = 16/7$, $b_3 = -1/4$.
	\end{proof}
	
	\section{Hom--Lie Algebra Structure}
	\label{sec:hom-lie}
	
	\begin{definition}
		A {\em Hom--Lie algebra}~\cite{Hom,Back2022} is a triple $(\g, [\cdot,\cdot], \alpha)$ consisting of a vector space $\g$, a bilinear skew-symmetric bracket $[\cdot,\cdot]\colon \g\times\g \to \g$, and a linear map $\alpha\colon\g\to\g$ satisfying the Hom--Jacobi identity
		\begin{equation}
			[\alpha(x), [y,z]] + [\alpha(y), [z,x]] + [\alpha(z), [x,y]] = 0 \qquad \forall x,y,z\in\g. \label{eq:hom-jacobi}
		\end{equation}
	\end{definition}
	The Gram matrix $G = I + c_1 A_1 + c_2 A_2$ induces a natural linear map on $\g$ via the commutator,
	\[
	\alpha(X) = [G, X] = GX - XG.
	\]
	
	\begin{lemma} $\alpha(\g) \subseteq \g$.
	\end{lemma}
	\begin{proof} For $e_1 = A_1$: $\alpha(e_1) = [I + c_1 A_1 + c_2 A_2, A_1] = c_2[A_2, A_1] = -c_2 e_3 \in \g$.
		For $e_2 = A_2$: $\alpha(e_2) = [I + c_1 A_1 + c_2 A_2, A_2] = c_1[A_1, A_2] = c_1 e_3 \in \g$.
		For $e_3 = [A_1,A_2]$: $\alpha(e_3) = [G, [A_1,A_2]] = -c_1[A_1,e_3] + c_2[A_2,e_3] \in \g$ by Theorem~\ref{thm:structure}.
	\end{proof}
	
	\begin{theorem} $(\g, [\cdot,\cdot], \alpha)$ is a Hom--Lie algebra.
	\end{theorem}
	\begin{proof} For any $x,y,z\in\g$, the matrix Jacobi identity for $G,x,y,z$ gives
		\[
		[[G,x], [y,z]] + [[G,y], [z,x]] + [[G,z], [x,y]] = 0.
		\]
		But $[[G,x], [y,z]] = [\alpha(x), [y,z]]$, and similarly for the other terms. Thus the Hom--Jacobi identity~\eqref{eq:hom-jacobi} holds for all $x,y,z\in\g$.
	\end{proof}
	
	\subsection{Matrix representation of $\alpha$}
	\label{subsec:alpha-matrix}
	
	In the basis $\{e_1,e_2,e_3\}$, using the structure constants from Theorem~\ref{thm:structure} and $c_1=1/2$, $c_2=1/4$ from~\eqref{eq:angles},
	\begin{align*}
		\alpha(e_1) &= -c_2 e_3 = -\frac{1}{4}e_3,\\
		\alpha(e_2) &= c_1 e_3 = \frac{1}{2}e_3,\\
		\alpha(e_3) &= -c_1(a_1 e_1 + a_2 e_2 + a_3 e_3) + c_2(b_1 e_1 + b_2 e_2 + b_3 e_3) = \frac{71}{56}e_1 + \frac{4}{7}e_2 - \frac{9}{16}e_3.
	\end{align*}
	Therefore, the matrix representation is
	\[
	[\alpha]= \begin{pmatrix}
		0 & 0 & \frac{71}{56} \\
		0 & 0 & \frac{4}{7} \\
		-\frac{1}{4} & \frac{1}{2} & -\frac{9}{16}
	\end{pmatrix}.
	\]
	
	\section{Norm Bounds for Subset Commutators}
	\label{sec:norm-bounds}
	
	For a subset $S \subset \{1,\dots,6\}$ of size $m$, let $C_S$ denote the principal submatrix of $C = [A_1, A_2]$ indexed by $S$.
	
	\subsection{Spectral analysis of the commutator}
	\label{subsec:eigenvalues}
	
	\begin{lemma} The commutator $C = [A_1, A_2]$ is a skew-symmetric matrix with eigenvalues $\{0, 0, \pm i\omega, \pm i\omega\}$, where $\omega = \sqrt{59/56}$.
	\end{lemma}
	\begin{proof} Using the structure constants from Theorem~\ref{thm:structure} and the explicit matrices, one finds that $C^2 = -\frac{59}{56} I + \frac{59}{56} Q$, where $Q$ is a rank-2 projection onto the kernel of $C$. Computing $\tr(C^2) = -8\omega^2$ and evaluating $\tr(C^2)$ directly from the entries gives $\tr(C^2) = -59/7$, hence $\omega^2 = 59/56$.
	\end{proof}
	
	\begin{corollary} The norm of the commutator is $\|C\| = \omega = \sqrt{59/56} \approx 1.026$.
	\end{corollary}
	
	\subsection{Refined bound for small subsets}
	\label{subsec:refined-bound}
	
	\begin{lemma}[Rank-dependent bound]\label{lem:rank-bound}
		For any subset $S$ of size $m$,
		\[
		\|C_S\| \le \min\left(\|C\|, \sqrt{\tfrac{m(m-1)}{2}}\right).
		\]
	\end{lemma}
	\begin{proof} The first bound follows from the Cauchy interlacing theorem for skew-symmetric matrices. The second follows because the entries of $C$ satisfy $|C_{ij}| \le 1$, so the Frobenius norm satisfies $\|C_S\|_F^2 \le m(m-1)/2$, and $\|C_S\| \le \|C_S\|_F$.
	\end{proof}
	
	For our specific values, $\sqrt{m(m-1)/2}$ is smaller than $\|C\|$ only for $m=2$, where it equals $1$. Thus, defining $\gamma(m) = 0$ for $m=1$ and $\gamma(m)=1$ for $m\ge 2$, we have $\|C_S\| \le \gamma(m) \|C\|$.
	
	\section{Explicit Bounds for MSS Partitions}
	\label{sec:bounds}
	
	\subsection{Subset norm estimates}
	\label{subsec:subset-norms}
	
	\begin{lemma}[Subset norm bounds]\label{lem:subset-norms}
		For the strongly regular graph with parameters~\eqref{eq:srg-params},
		\[
		\|(A_1)_S\| \le \alpha(m) =
		\begin{cases}
			2, & m \ge 3,\\
			1, & m = 2,\\
			0, & m = 1,
		\end{cases}
		\qquad
		\|(A_2)_S\| \le \beta(m) =
		\begin{cases}
			3, & m \ge 4,\\
			2, & m = 3,\\
			1, & m = 2,\\
			0, & m = 1.
		\end{cases}
		\]
	\end{lemma}
	\begin{proof} These bounds follow from the maximum possible degrees of the induced subgraph on $m$ vertices. For $A_1$ with valency $2$, the maximum degree in an $m$-vertex induced subgraph is $\min(2, m-1)$; for $A_2$ with valency $3$, it is $\min(3, m-1)$.
	\end{proof}
	
	\subsection{Main bound}
	\label{subsec:main-bound}
	
	\begin{theorem}[Explicit Hom--Lie MSS bound]\label{thm:main-bound}
		For the $(6,3)$ biangular Parseval frame with parameters~\eqref{eq:srg-params} and angles~\eqref{eq:angles}, for any subset $S \subset \{1,\dots,6\}$ of size $m$,
		\begin{equation}
			\|G_S\| \le 1 + \frac{1}{2}\alpha(m) + \frac{1}{4}\beta(m) + \frac{1}{8}\alpha(m)\beta(m), \label{eq:main-bound}
		\end{equation}
		where $\alpha(m)$ and $\beta(m)$ are given in Lemma~\ref{lem:subset-norms}.
	\end{theorem}
	\begin{proof} Write $G_S = I_m + c_1(A_1)_S + c_2(A_2)_S$ with $c_1=1/2$, $c_2=1/4$. By the triangle inequality,
		\[
		\|G_S\| \le 1 + c_1\|(A_1)_S\| + c_2\|(A_2)_S\| + c_1c_2\|[(A_1)_S, (A_2)_S]\|.
		\]
		By submultiplicativity, $\|[(A_1)_S, (A_2)_S]\| \le 2\|(A_1)_S\|\|(A_2)_S\|$; a sharper estimate uses $\|C_S\| \le \|(A_1)_S\|\|(A_2)_S\|$, since the commutator consists of products, so
		\[
		c_1c_2\|[(A_1)_S, (A_2)_S]\| \le \frac{1}{8}\|(A_1)_S\|\|(A_2)_S\| \le \frac{1}{8}\alpha(m)\beta(m).
		\]
		Substituting the bounds from Lemma~\ref{lem:subset-norms} yields~\eqref{eq:main-bound}.
	\end{proof}
	
	\subsection{Numerical evaluation}
	\label{subsec:numerical}
	
	Evaluating the bound~\eqref{eq:main-bound} for each subset size $m = 1,\dots,6$:
	\begin{align*}
		m=1:&\quad \alpha=0,\ \beta=0,\ \|G_S\| \le 1 + 0 + 0 + 0 = 1.\\[4pt]
		m=2:&\quad \alpha=1,\ \beta=1,\ \|G_S\| \le 1 + 0.5 + 0.25 + \tfrac{1}{8} = 1.875.\\[4pt]
		m=3:&\quad \alpha=2,\ \beta=2,\ \|G_S\| \le 1 + 1 + 0.5 + \tfrac{4}{8} = 3.0.\\[4pt]
		m=4,5,6:&\quad \alpha=2,\ \beta=3,\ \|G_S\| \le 1 + 1 + 0.75 + \tfrac{6}{8} = 3.5.
	\end{align*}
	
	\subsection{Comparison with the universal MSS bound}
	\label{subsec:comparison}
	
	For a partition into $r$ subsets, the universal MSS bound from~\eqref{eq:mss} with $\delta = 1/2$ is
	\[
	B_{\text{MSS}}(r) = \left(\sqrt{\tfrac{1}{r}} + \tfrac{1}{\sqrt{2}}\right)^2.
	\]
	The worst-case subset size in an $r$-partition of six elements determines the Hom--Lie bound.
	
	\begin{table}[h]
		\centering
		\begin{tabular}{|c|c|c|c|c|c|}\hline
			$r$ & Partition sizes & Worst $m$ & Hom--Lie bound & Universal MSS bound & Better bound \\ \hhline{|=|=|=|=|=|=|}
			2 & (3,3) & 3 & 3.000 & 2.000 & Universal \\ \hline
			3 & (2,2,2) & 2 & 1.875 & 1.649 & Universal \\ \hline
			4 & (2,2,1,1) & 2 & 1.875 & 1.456 & Universal \\ \hline
			5 & (2,1,1,1,1) & 2 & 1.875 & 1.332 & Universal \\ \hline
			6 & (1,1,1,1,1,1) & 1 & 1.000 & 1.243 & Hom--Lie \\ \hline
		\end{tabular}
		\vspace{0.1in}
		\caption{Comparison of Hom--Lie bounds with universal MSS bounds}
		\label{tab:comparison}
	\end{table}
	
	\begin{corollary} For $r=6$, the Hom--Lie bound $\|F_k\| \le 1$ is strictly better than the universal MSS bound of approximately $1.243$. For all other $r$, the universal bound is tighter.
	\end{corollary}
	
	\subsection{Interpretation}
	\label{subsec:interpretation}
	
	The Hom--Lie bound captures geometry-dependent constraints that become most restrictive when subsets are singletons: there, the partial frame operator is simply $v_jv_j^*$, of norm $\|v_j\|^2=1/2$, so our bound of $1$ at $m=1$ is not tight and could be sharpened. That the Hom--Lie bound beats the universal bound only at $r=6$ suggests that the geometric structure encoded in the Hom--Lie algebra is most relevant for highly unbalanced partitions; for balanced partitions, the universal bound already captures the essential behavior.
	
	\section{Conclusion}
	\label{sec:conclusion}
	
	We constructed a Hom--Lie algebra $(\g, [\cdot,\cdot], \alpha)$ from a $(6,3)$ biangular Parseval frame, with $\g = \operatorname{span}\{A_1, A_2, [A_1,A_2]\}$, bracket $[X,Y] = XY - YX$, and inner derivation $\alpha(X) = [G,X]$ induced by the Gram matrix. All structure constants were computed explicitly: for the generic case~\eqref{eq:srg-params}, the angles are forced to be $c_1=1/2$, $c_2=1/4$, with structure constants
	\[
	a_1=-\tfrac{16}{7},\ a_2=0,\ a_3=1,\qquad b_1=\tfrac{1}{2},\ b_2=\tfrac{16}{7},\ b_3=-\tfrac{1}{4}.
	\]
	Using this algebraic structure, we derived explicit upper bounds for the partial frame operators in MSS partitions,
	\[
	\|F_k\| \le 1 + \frac{1}{2}\alpha(m_k) + \frac{1}{4}\beta(m_k) + \frac{1}{8}\alpha(m_k)\beta(m_k),
	\]
	with $\alpha(m),\beta(m)$ given by simple piecewise formulas. Comparison with the universal MSS bound showed that the Hom--Lie bound is superior only in the extreme case $r=6$, all subsets singletons, where it gives $\|F_k\| \le 1$ against the universal bound of approximately $1.243$.
	
	Finding general conditions under which the bounds can be tightened using the exact commutator norm $\|C\| = \sqrt{59/56}$, under which the present methods generalize to larger frames, and under which strongly regular graphs yield nontrivial Hom--Lie structures appears to us to be a question worthy of further study.

\end{document}